\documentclass[11pt]{article}

\usepackage{amsmath,amssymb}
\usepackage{amsthm}
\usepackage{graphicx}  

\usepackage{stmaryrd}
\usepackage{geometry}
\usepackage[colorlinks=true]{hyperref}
 \newcommand{\eps}{\varepsilon}

\begin{document}

\theoremstyle{plain}
\newtheorem{Lemma}{Lemma}
\newtheorem{Theorem}{Theorem}
\newtheorem{Proposition}{Proposition}
\newtheorem{Corollary}{Corollary}
\newtheorem{Conjecture}{Conjecture}
\newtheorem{Problem}{Problem}
\newtheorem*{Problem*}{Problem}
\newtheorem*{Openproblem*}{Open problem}
\newtheorem{Remark}{Remark}

\title{The distance problem on 
measured metric spaces} 
\author{David J. Aldous\thanks{Department of Statistics,
 367 Evans Hall \#\  3860,
 U.C. Berkeley CA 94720; \href{mailto:aldousdj@berkeley.edu}{aldousdj@berkeley.edu};
\href{https://www.stat.berkeley.edu/users/aldous/}{https://www.stat.berkeley.edu/users/aldous/}.}
  \and 
Guillaume Blanc\thanks{Universit\'e Paris-Saclay; \href{mailto:guillaume.blanc1@universite-paris-saclay.fr}{guillaume.blanc1@universite-paris-saclay.fr}; \href{https://sites.google.com/view/guillaume-blanc-math}{https://sites.google.com/view/guillaume-blanc-math}}
  \and
Nicolas Curien\thanks{Universit\'e Paris-Saclay; \href{mailto:nicolas.curien@gmail.com}{nicolas.curien@gmail.com}; \href{https://www.imo.universite-paris-saclay.fr/~nicolas.curien/}{https://www.imo.universite-paris-saclay.fr/~nicolas.curien/}}
}


\maketitle
 
\begin{abstract}
What distributions arise as the distribution of the distance between two typical points in some measured metric space?
This seems to be a surprisingly subtle problem.
We conjecture that every distribution with a density function 
whose support contains $0$ does arise in this way,
and give some partial results in that direction.
\end{abstract}
 
\section{Introduction and main results}

The problem below has apparently not previously been studied.
It seems quite natural in itself, and some statistical motivation is given towards the end of the introduction. 
Let $\mu$ be a Borel probability measure on a complete \textbf{separable} metric space $(S,d)$, and let $\xi_1$ and $\xi_2$ be independent random variables with distribution $\mu$.
The random variable $D:=d(\xi_1,\xi_2)$ has some distribution, say $\theta$, on $\mathbb{R}_+$.
We call a distribution on $\mathbb{R}_+$ \textbf{feasible} if it arises in this way, and we say that $(S,d,\mu)$ \textbf{achieves} $\theta$.

\begin{Problem}\label{prob:main}
Describe the set of feasible distributions.
\end{Problem}

The results stated below are essentially all we know about this problem.
Before coming to them, let us stress that the separability requirement is fundamental.
For example, it guarantees that $D$ is a random variable, since in this setting the Borel $\sigma$-algebra $\mathcal{B}(S\times S)$ agrees with $\mathcal{B}(S)\otimes\mathcal{B}(S)$.
Besides, the construction of interesting Borel measures on non-separable metric spaces is problematic: for instance, it is known that $(S,d)$ must be separable as soon as it supports a boundedly finite measure $\mu$ of full support \cite[Theorem 4.1]{gaczkowskigorka}.


In our problem, a first consequence of the separability assumption is that $\theta$ must have $0$ in its support (Proposition \ref{prop:0supp}).
But Proposition \ref{prop:mathoverflow} will show that the converse is false: there are distributions on $\mathbb{R}_+$ with $0$ in their support that are not feasible. 

\begin{Proposition}\label{prop:0supp}
Any feasible distribution must have $0$ in its support.
\end{Proposition}
Informally, in nice homogeneous settings, we expect that 
\[\theta[0,\varepsilon]=\mathbb{P}(D \leq \varepsilon) \approx  \varepsilon^{ \mathrm{dim}}\quad\text{as $\varepsilon\to0^+$,}\]
where $ \mathrm{dim}$ is the ``dimension'' of $(S,d)$. 
See Remark \ref{rek:dim} for a precise statement in this direction.
This indicates that if $\theta$ assigns small mass around $0$, then $(S,d)$ must be ``large'', and indeed Proposition \ref{prop:0supp} is a caricature of this fact: if $\theta$ assigns no mass around $0$, then $(S,d)$ must be so large that it cannot be separable. 
There is however no obstruction in general:

\begin{Proposition}\label{prop:petit0}
For every non-decreasing function $F:\mathbb{R}_+^*\rightarrow{]0,1]}$, it is possible to construct a compact measured metric space $(S,d,\mu)$ such that $\mathbb{P}(D\leq\varepsilon)\leq F(\varepsilon)$ for all sufficiently small $\varepsilon$.
\end{Proposition}
Our notation here is
\[ \mathbb{R}_+ := [0,\infty[ \quad \ \mathbb{R}_+^* := ]0,\infty[ \quad \mathbb{N} := \{0,1,2, \ldots \} \quad \ \mathbb{N}^* := \{ 1,2, \ldots \}. \]

Moving away from the behavior at $0^+$, a common first thought on the problem is that the triangle inequality puts some constraint on the general structure of a feasible $\theta$. 
Perhaps surprisingly, we first show that any distribution on $\mathbb{R}_+$ whose support is a finite set containing $0$ is feasible:

\begin{Theorem}\label{thm:finite}
For every distribution of the form $\theta=p_0\cdot\delta_0+p_1\cdot\delta_{d_1}+\ldots+p_k\cdot\delta_{d_k}$, where $k\in\mathbb{N}$ and ${0=d_0<d_1<\ldots<d_k}$, and where $p_0,p_1,\ldots,p_k\in{]0,1[}$ are such that ${p_0+p_1+\ldots+p_k=1}$, there exists a compact measured metric space $(S,d,\mu)$ that achieves $\theta$.
\end{Theorem}

Our proof of Theorem \ref{thm:finite} relies on a tree construction.
Although it might be appealing to approximate a target distribution on $\mathbb{R}_+$ by finitely supported distributions $\theta_n$, each of which is achieved by some $(S_n,d_n,\mu_n)$, and hope to get something in the limit $n\to\infty$, we discuss in Remark \ref{rem:limit} why this is not possible in general. 
On the other hand, the tree constructions of Theorem \ref{thm:finite} and Proposition \ref{prop:petit0} can be used to prove the following result:

\begin{Proposition}\label{prop:?}
Let $f$ be a continuous probability density function on $\mathbb{R}_+$, and assume that there exists $\eta>0$ such that $f(x)>0$ for all $x\in{]0,\eta]}$.
Then, for every $\zeta>0$, there exists a compact measured metric space $(S,d,\mu)$ such that the distribution $\theta$ arising from $(S,d,\mu)$ has a continuous density $g$ on $\mathbb{R}_+$, where $g\leq(1+\zeta)\cdot f$.
\end{Proposition}

A weaker result is that, for such $f$, there exist feasible densities $g_n$ such that $\int |g_n(x) - g(x)| dx \to 0$.
This weaker form could be derived quite easily from Theorem  \ref{thm:finite}, by replacing each leaf of the tree by a short interval. 
The specific stronger form in Proposition \ref{prop:?} requires a more elaborate construction.
The form is motivated by the following observation: 
If there were a space $(S^\prime,d,\mu^\prime)$ achieving the distribution $\theta$ with density $f$,
 then take a compact set $K$ with $\mu^\prime(K) \ge 1 - \eps$, and then 
 the {\bf compact} space $ (K,d, \mu^\prime(\cdot | K))$ achieves a distribution $\theta^*$ 
 with $\theta^* \le (1-\eps)^{-2} \theta$.


A straightforward corollary of Proposition \ref{prop:?} is that every distribution with a suitable density is achieved by some \textbf{random} compact measured metric space (Corollary \ref{cor:density}).
In this setting, we first sample the random compact measured metric space $(S,d,\mu)$, then conditionally on $(S,d,\mu)$, we sample two independent random variables $\xi_1$ and $\xi_2$ with distribution $\mu$, and we consider the annealed distribution of $d(\xi_1,\xi_2)$.
See Section \ref{sec:csq} for details.

\begin{Corollary}\label{cor:density}
Let $f$ be a continuous probability density function on $\mathbb{R}_+$, and assume that there exists $\eta>0$ such that $f(x)>0$ for all $x\in{]0,\eta]}$.
Then, there exists a \textbf{random} compact measured metric space $(S,d,\mu)$ that achieves the distribution with density $f$.
\end{Corollary}

We highlight that the following problem remains open.

\begin{Openproblem*}
Prove that for every probability density function $f$ on $\mathbb{R}_+$ whose support contains $0$, the distribution with density $f$ is feasible.
\end{Openproblem*}

\paragraph{Background and motivation.} 
We finish this introduction with some statistical motivation for the problem. Consider a Borel probability measure $\mu$ on a complete separable metric space $(S,d)$, and let $\xi_1,\xi_2,\ldots$ be a sequence of independent random variables with distribution $\mu$.
In this setting, it is known that the distribution of the infinite array $A_\infty=(d(\xi_i,\xi_j)\,;\,i,j\geq1)$ determines $(S,d,\mu)$ up to measure-preserving isometry (we refer to \cite{vershik,vershik2}, see also \cite[Chapter 3$\frac{1}{2}_+$]{gromov}), although there is no known explicit characterization\footnote{This setting contrasts with the case where $d$ ranges over all {\em measurable} functions, in which case the analogous infinite arrays are characterized by exchangeability properties: see \cite{PE1,PE2,PE3} and a more recent nice account  by Tim Austin \cite{austin}.}  of the possible distributions of $A_\infty$. See Remark \ref{rem:limit} for the connections with Theorem \ref{thm:finite}.
For each $n \geq2$, the distribution  of 
\begin{equation}
A_n := (d(\xi_i,\xi_j)\,;\,1 \le i,j  \le n),
\label{An}
\end{equation}
is usually called the \textbf{$n$-point function/distribution}. Describing the possible distributions of $A_{\infty}$ is equivalent to understanding, for each $n\geq2$, what constraints on the $(n+1)$-point function are given by the $n$-point function.
The fact that we do not know which distributions are $2$-point functions (Problem \ref{prob:main}) makes it hard to proceed, and  our work  is a modest start.
After having described the set of feasible distributions, one could continue to ask for the possible $3$-point functions, where now the triangle inequality would come into play...
 
Now, here is the statistical modelling context.
Suppose that we have a large database of different objects of the same type\footnote{E.g, fingerprints, human DNA (in the forensic context), facial recognition, musical tunes or lyrics (there are of the order of 100 million songs online: \href{https://www.musicianwave.com/how-many-songs-are-there-in-the-world/}{https://www.musicianwave.com/how-many-songs-are-there-in-the-world/}),  the plot of your new murder mystery novel (in the copyright context)...}, and we want to decide whether a new object is significantly similar to some object in the database -- more similar than would be expected ``by chance".  A natural model in this general context is that there is a space $(S,d)$ of possible objects with distances, and that our database objects and the new object are i.i.d. samples $(\xi_i,\,i \geq 1)$ from a probability measure $\mu$ on $S$. However, we do not observe $S$ or $\mu$, all we observe are the distances $A_{n}= (d( \xi_{i}, \xi_{j})\,;\,1 \leq i,j \leq n)$ between these objects.
One would like to devise an algorithm that, given $(d(\xi_{n+1},\xi_i)\,;\,1\leq i \leq n)$, decides whether $\xi_{n+1}$ is ``too close to one of $\{ \xi_{1}, ... , \xi_{n}\}$ to just be chance", which would then suggest some causal relationship. 
In this context, we seek to make inferences which are ``universal'', i.e, do not depend on $(S, d, \mu)$, and this motivates the study of relationships between $n$-point functions.

\section{Properties of the support of feasible distributions}

As mentioned in the introduction, a first consequence of the separability assumption is that any feasible distribution must have $0$ in its support (Proposition \ref{prop:0supp}).
Let us prove this now.

\begin{proof}[Proof of Proposition \ref{prop:0supp}]
Let $\{s_n,\,n\geq1\}$ be a dense countable subset of $S$.
For each $\varepsilon>0$, since the balls $(B(s_n,\varepsilon/2)\,;\,n\geq1)$ cover $S$, there exists $n\in\mathbb{N}^*$ such that $\mu(B(s_n,\varepsilon/2))>0$.
It follows that 
\[\mathbb{P}(D\leq\varepsilon)\geq\mu(B(s_n,\varepsilon/2))^2>0.\]
\end{proof}

\begin{Remark} \label{rek:dim}
If $(S,d)$ is compact, then the previous argument can be sharpened to show that
\begin{equation}\label{eq:minkowski}
\mathbb{P}(D\leq\varepsilon)\geq M(\varepsilon/2)^{-1},
\end{equation}
where $M(\varepsilon/2)$ is the minimal number of balls with radius $\varepsilon/2$ needed to cover $S$.
Indeed, fix a covering $(B(s_i,\varepsilon/2)\,;\,i\in\llbracket1,k\rrbracket)$ of $S$ by 
balls\footnote{Notation $\llbracket 1,k  \rrbracket$ indicates the interval of integers.}
 of radius $\varepsilon/2$.
Then, let
\[B_i=B(s_i,\varepsilon/2)\setminus(B(s_1,\varepsilon/2)\cup\ldots\cup B(s_{i-1},\varepsilon/2))\quad\text{for all $i\in\llbracket1,k\rrbracket$,}\]
so that $B_1\cup\ldots\cup B_k=S$, where the $B_i$ have diameter at most $\varepsilon$.
On the one hand, we have
\[\mathbb{P}(D\leq\varepsilon)\geq\mathbb{P}\left(\bigsqcup_{i=1}^k(\xi_1,\xi_2\in B_i)\right)=\sum_{i=1}^k\mu(B_i)^2.\]
On the other hand, using the Cauchy--Schwarz inequality, we have
\[1=\mu\left(\bigsqcup_{i=1}^kB_i\right)^2=\left(\sum_{i=1}^k\mu(B_i)\right)^2\leq k\cdot\sum_{i=1}^k\mu(B_i)^2.\]
Thus, we get
\[\mathbb{P}(D\leq\varepsilon)\geq k^{-1}.\]
Taking $k$ to be minimal proves \eqref{eq:minkowski}.
\end{Remark}

The next example\footnote{See {\em Acknowledgements} for the origin of this example.} shows that having $0$ in its support is not a {\em sufficient} condition for a distribution on $\mathbb{R}_+$ to be feasible.

\begin{Proposition}\label{prop:mathoverflow}
Any distribution of the form $\theta=p\cdot\delta_0+(1-p)\cdot\theta'$, where $p\in{]0,1[}$, and where $\theta'$ is a non-atomic distribution on $\mathbb{R}_+$ whose support does not contain $0$, is not feasible.
\end{Proposition}
\begin{proof}
We prove this by contraposition.
First, if $\mu\{s\}=0$ for all $s\in S$, then by Fubini's theorem, we have
\[\mathbb{P}(D=0)=\int_S\int_S\mathbf{1}(s_1=s_2)\mathrm{d}\mu(s_2)\mathrm{d}\mu(s_1)=\int_S\mu\{s_1\}\mathrm{d}\mu(s_1)=0.\]
Next, if there exists $s_1\neq s_2\in S$ such that $\mu\{s_1\}>0$ and $\mu\{s_2\}>0$, then we have $d(s_1,s_2)>0$, and
\[\mathbb{P}(D=d(s_1,s_2))\geq\mathbb{P}(\xi_1=s_1\,;\,\xi_2=s_2)=\mu\{s_1\}\cdot\mu\{s_2\}>0.\]
Finally, if $\mu\{s_0\}>0$ for exactly one $s_0\in S$, then we have the following alternatives.
\begin{itemize}
\item If $\mu\{s_0\}=1$, then $D=0$ almost surely.
\item If $\mu\{s_0\}<1$, then $\mu'=(1-\mu\{s_0\})^{-1}\cdot(\mu-\mu\{s_0\}\cdot\delta_{s_0})$ is a probability measure such that $\mu'\{s\}=0$ for all $s\in S$.
By the first case treated above, the distribution $\theta'$ arising from $(S,d,\mu')$ is such that $\theta'\{0\}=0$.
On the other hand, by Proposition \ref{prop:0supp}, we have $\theta'[0,\varepsilon]>0$ for all $\varepsilon>0$.
It follows that
\[\mathbb{P}(D\leq\varepsilon)\geq\mu\{s_0\}^2+(1-\mu\{s_0\})^2\cdot\theta'[0,\varepsilon]>\mu\{s_0\}^2=\mathbb{P}(D=0).\]
\end{itemize}
In none of the cases treated above (which cover every possibility) the distribution of $D$ has the form stated in the proposition.
\end{proof}

\section{The tree constructions}

In this section we use tree constructions to prove our main results, Theorem \ref{thm:finite} and Proposition \ref{prop:petit0}.
For the proof of Theorem \ref{thm:finite}, the measured metric spaces $(S,d,\mu)$ we construct to achieve finitely supported distributions are finite rooted trees $(T,\rho)$ with edge-lengths, equipped with the natural metric $d$ induced by the edge-lengths, and endowed with a probability measure $\nu$ supported on the leaves.
We call the tuple $(T,\rho,d,\nu)$ a \emph{tree structure} for short.
Also, recall that we use ``$(T,\rho,d,\nu)$ achieves $\theta$'' as shorthand for ``if $\xi_1$ and $\xi_2$ are independent random variables with distribution $\nu$, then $d(\xi_1,\xi_2)$ has distribution $\theta$''.
Theorem \ref{thm:finite} is implied by:

\begin{Proposition}\label{prop:treeconstruction}
For any distribution of the form $\theta=p_0\cdot\delta_0+p_1\cdot\delta_{d_1}+\ldots+p_k\cdot\delta_{d_k}$, where $k\in\mathbb{N}$ and $0=d_0<d_1<\ldots<d_k$, and where $p_0,p_1,\ldots,p_k\in{]0,1[}$ are such that ${p_0+p_1+\ldots+p_k=1}$, there exists a tree structure $(T,\rho,d,\nu)$ that achieves $\theta$.
\end{Proposition}
\begin{proof}
Let us prove, by induction on $k\in\mathbb{N}$, the more detailed assertion 
\begin{quote}
$H_k$: ``for any distribution of the form $\theta=p_0\cdot\delta_0+p_1\cdot\delta_{d_1}+\ldots+p_k\cdot\delta_{d_k}$, where $0=d_0<d_1<\ldots<d_k$, and where $p_0,p_1,\ldots,p_k\in{]0,1[}$ are such that ${p_0+p_1+\ldots+p_k=1}$, there exists a tree structure $(T,\rho,d,\nu)$ that achieves $\theta$, in which every leaf is at distance $d_k/2$ from the root vertex $\rho$''.
\end{quote}
The construction is illustrated in Figure \ref{Fig:1}.
Note that it is ``backwards'', in that the length of the edges emanating from $\rho$ is $(d_k-d_{k-1})/2$, not $d_1/2$.

\setlength{\unitlength}{0.19in}
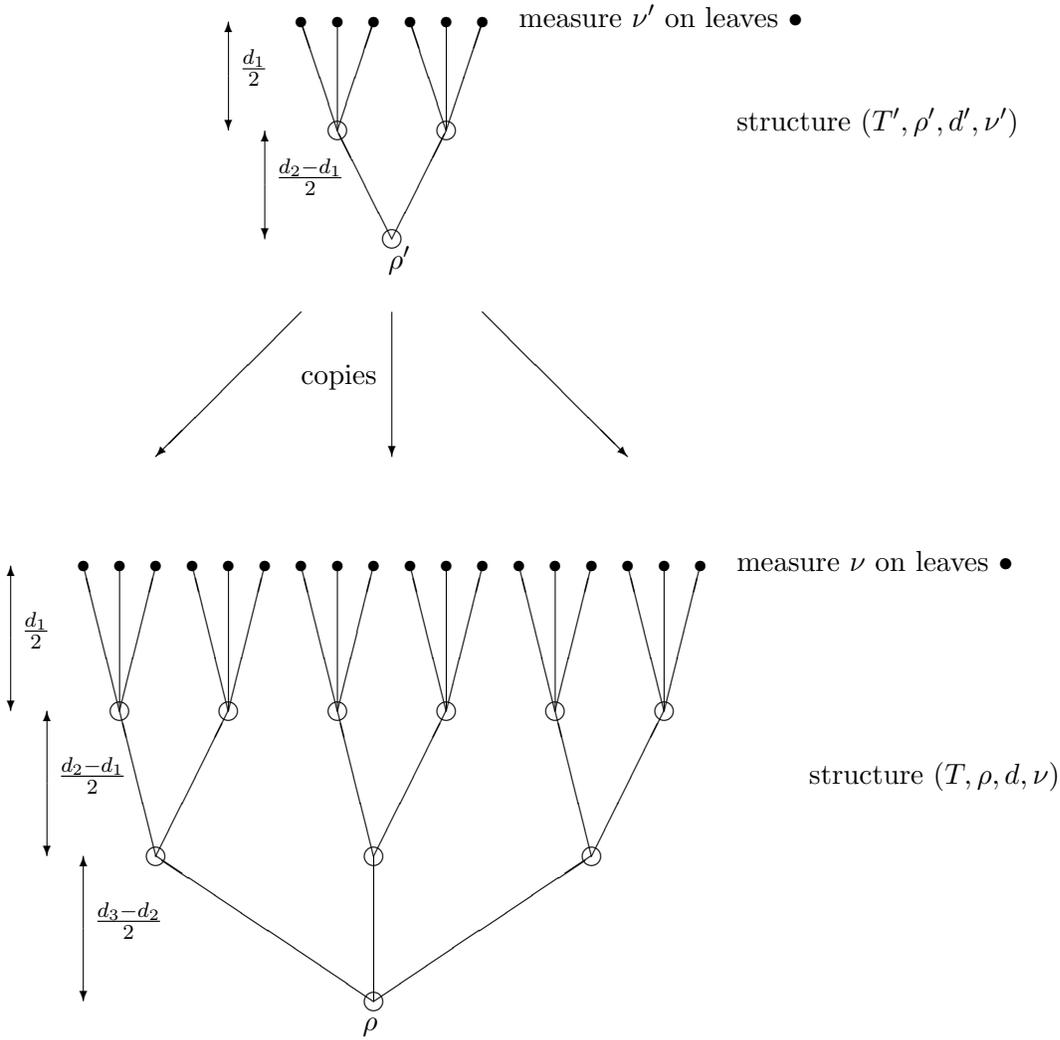
\begin{figure}[h!]
\begin{center}
\begin{picture}(30,33)(0,5)

\multiput(6,33)(1,0){6}{\circle*{0.3}}
\multiput(7,30)(3,0){2}{\circle{0.5}}
\put(8.5,27){\circle{0.5}}
\put(8.4,26.2){$\rho^\prime$}

\multiput(7,30)(3,0){2}{\line(-1,3){1}}
\multiput(7,30)(3,0){2}{\line(1,3){1}}
\multiput(7,30)(3,0){2}{\line(0,3){3}}

\put(8.5,27){\line(-1,2){1.5}}
\put(8.5,27){\line(1,2){1.5}}

\multiput(0,18)(1,0){18}{\circle*{0.3}}
\multiput(1,14)(3,0){6}{\circle{0.5}}
\multiput(2,10)(6,0){3}{\circle{0.5}}
\put(8,6){\circle{0.5}}
\put(7.7,5.2){$\rho$}

\multiput(1,14)(3,0){6}{\line(-1,4){1}}
\multiput(1,14)(3,0){6}{\line(1,4){1}}
\multiput(1,14)(3,0){6}{\line(0,4){4}}

\multiput(2,10)(6,0){3}{\line(-1,4){1}}
\multiput(2,10)(6,0){3}{\line(1,2){2}}

\put(8,6){\line(0,1){4}}
\put(8,6){\line(-3,2){6}}
\put(8,6){\line(3,2){6}}

\put(6,25){\vector(-1,-1){4}}
\put(11,25){\vector(1,-1){4}}
\put(8.5,25){\vector(0,-1){4}}
\put(6,23){copies}

\put(-2,16){\vector(0,1){2}}
\put(-2,16){\vector(0,-1){2}}
\put(-1.7,16){$\frac{d_1}{2}$}

\put(-1,12){\vector(0,1){2}}
\put(-1,12){\vector(0,-1){2}}
\put(-0.7,12){$\frac{d_2-d_1}{2}$}

\put(0,8){\vector(0,1){2}}
\put(0,8){\vector(0,-1){2}}
\put(0.3,8){$\frac{d_3-d_2}{2}$}

\put(5,28.5){\vector(0,1){1.5}}
\put(5,28,5){\vector(0,-1){1.5}}
\put(5.3,28.5){$\frac{d_2-d_1}{2}$}

\put(4,31.5){\vector(0,1){1.5}}
\put(4,31,5){\vector(0,-1){1.5}}
\put(4.3,31.5){$\frac{d_1}{2}$}

\put(12,32.85){measure $\nu'$ on leaves $\bullet$}
\put(18,17.85){measure $\nu$ on leaves $\bullet$}

\put(18,30){structure $(T',\rho',d',\nu')$}
\put(20,12){structure $(T,\rho,d,\nu)$}

\end{picture}

\caption{Illustrating construction of the tree structure $T = (T,\rho,d,\nu)$ for $n = 3$.}
\label{Fig:1}
\end{center}
\end{figure}

To prove $H_0$, consider the tree structure $(T,\rho,d,\nu)$ consisting of the single root vertex $\rho$.
Then, fix $k\in\mathbb{N}$, assume that $H_k$ holds, and let us prove $H_{k+1}$.
Let $\theta$ be a distribution of the form ${\theta=p_0\cdot\delta_0+p_1\cdot\delta_{d_1}+\ldots+p_{k+1}\cdot\delta_{d_{k+1}}}$, where ${0=d_0<d_1<\ldots<d_{k+1}}$, and where $p_0,p_1,\ldots,p_{k+1}\in{]0,1[}$ are such that ${p_0+p_1+\ldots+p_{k+1}=1}$.
We can write 
\[\theta=(1-p_{k+1})\cdot\theta'+p_{k+1}\cdot\delta_{d_{k+1}},\]
where
\[\theta'=\frac{p_0}{p_0+\ldots+p_k}\cdot\delta_0+\frac{p_1}{p_0+\ldots+p_k}\cdot\delta_{d_1}+\ldots+\frac{p_k}{p_0+\ldots+p_k}\cdot\delta_{d_k}.\]
Now, by $H_k$, there exists a tree structure $(T',\rho',d',\nu')$ that achieves $\theta'$, in which every leaf is at distance $d_k/2$ from the root vertex $\rho'$.
Then, let $j\in\mathbb{N}^*$ be large enough so that we can choose $m_1,\ldots,m_j\in{]0,1[}$ with ${m_1+\ldots+m_j=1}$, such that $m_1^2+\ldots+m_j^2=1-p_{k+1}$.
Take $j$ copies $(T_1,\rho_1,d_1,\nu_1),\ldots,(T_j,\rho_j,d_j,\nu_j)$ of $(T',\rho',d',\nu')$, and construct a tree structure  $(T,\rho,d,\nu)$ by first drawing $j$ edges of length ${(d_{k+1}-d_k)/2}$ emanating from a root vertex $\rho$, then grafting the $(T_i,\rho_i,d_i,\nu_i)$ onto those edges, identifying $\rho_i$ with the end-vertex of the corresponding edge, and letting $\nu=m_1\cdot\nu_1+\ldots+m_j\cdot\nu_j$.
By construction, the probability measure $\nu$ is supported on the leaves of $T$, and every leaf of $T$ is at distance ${d_k/2+(d_{k+1}-d_k)/2=d_{k+1}/2}$ from the root vertex $\rho$.
Now, let us check that $(T,\rho,d,\nu)$ achieves $\theta$.
Let $\xi_1$ and $\xi_2$ be independent random variables with distribution $\nu$.
For every Borel function $\varphi:\mathbb{R}_+\rightarrow\mathbb{R}_+$, we calculate
\[\begin{split}
\mathbb{E}[\varphi(d(\xi_1,\xi_2))]&=\sum_\text{$u_1,u_2$ leaves of $T$}\varphi(d(u_1,u_2))\cdot\nu\{u_1\}\cdot\nu\{u_2\}\\
&=\sum_{1\leq i_1,i_2\leq j}m_{i_1}\cdot m_{i_2}\cdot\sum_{\substack{\text{$u_1$ leaf of $T_{i_1}$}\\\text{$u_2$ leaf of $T_{i_2}$}}}\varphi(d(u_1,u_2))\cdot\nu_{i_1}\{u_1\}\cdot\nu_{i_2}\{u_2\}.
\end{split}\]
In the sum above, if $i_1=i_2$, then
\begin{eqnarray*}
\lefteqn{\sum_{\substack{\text{$u_1$ leaf of $T_{i_1}$}\\\text{$u_2$ leaf of $T_{i_2}$}}}\varphi(d(u_1,u_2))\cdot\nu_{i_1}\{u_1\}\cdot\nu_{i_2}\{u_2\}}\\
&=&\sum_{\text{$u_1,u_2$ leaves of $T'$}}\varphi(d(u_1,u_2))\cdot\nu'\{u_1\}\cdot\nu'\{u_2\}=\theta'(\varphi),
\end{eqnarray*}
by the definition of $(T',\rho',d',\nu')$.
On the other hand, if $i_1\neq i_2$, then the distance between any leaf $u_1$ of $T_{i_1}$ and any leaf $u_2$ of $T_{i_2}$ in the tree structure $(T,\rho,d,\nu)$ is 
\[d_k/2+(d_{k+1}-d_k)/2+(d_{k+1}-d_k)/2+d_k/2=d_{k+1},\]
hence
\begin{eqnarray*}
\lefteqn{\sum_{\substack{\text{$u_1$ leaf of $T_{i_1}$}\\\text{$u_2$ leaf of $T_{i_2}$}}}\varphi(d(u_1,u_2))\cdot\nu_{i_1}\{u_1\}\cdot\nu_{i_2}\{u_2\}}\\
&=&\sum_{\substack{\text{$u_1$ leaf of $T_{i_1}$}\\\text{$u_2$ leaf of $T_{i_2}$}}}\varphi(d_{k+1})\cdot\nu_{i_1}\{u_1\}\cdot\nu_{i_2}\{u_2\}=\varphi(d_{k+1}).
\end{eqnarray*}
It follows that
\[\begin{split}
\mathbb{E}[\varphi(d(\xi_1,\xi_2))]&=\sum_{i=1}^jm_i^2\cdot\theta'(\varphi)+\sum_{1\leq i_1\neq i_2\leq j}m_{i_1}\cdot m_{i_2}\cdot\varphi(d_{k+1})\\
&=(1-p_{k+1})\cdot\theta'(\varphi)+p_{k+1}\cdot\varphi(d_{k+1})=\theta(\varphi).
\end{split}\]
\end{proof}

\begin{Remark}\label{rem:limit} 
It is natural, for a given target distribution $\theta$, say with a smooth and compactly supported density, to approximate $\theta$ by distributions $\theta_n$ whose support is a finite set containing $0$, in such a way that $\theta_n\Rightarrow\theta$ as $n\to\infty$.
By Theorem \ref{thm:finite}, each $\theta_{n}$ is achieved by some tree structure $(T_n,\rho_n,d_n,\nu_n)$, and it is natural to seek for some (sub-)sequential limits of $(T_n,\rho_n,d_n,\nu_n)_{n\in\mathbb{N}^*}$ to achieve $\theta$. 
Alas, with the construction presented above, in general the sequence $(T_n,\rho_n,d_n,\nu_n)_{n\in\mathbb{N}^*}$ is not tight for the Gromov--Hausdorff--Prokhorov topology (see Section \ref{sec:csq} for a brief reminder on the Gromov--Hausdorff--Prokhorov topology).
Intuitively, the spaces we would end up with ``in the limit'' are rather non-separable. 
A cartoon of the phenomenon is the non-convergence of $n$-star graphs ($n$ vertices, each connected to the same root vertex by an edge of length $1$, with the uniform probability measure on the $(n+1)$ vertices) for the Gromov--Hausdorff Prokhorov topology, although for each $k\geq2$, their $k$-point functions converge in distribution as $n\to\infty$.
\end{Remark}

We now turn to the proof of Proposition \ref{prop:petit0}, which is again based on a tree construction.

\begin{proof}[Proof of Proposition \ref{prop:petit0}]
Let $(\kappa_n)_{n\in\mathbb{N}}$ be a sequence of positive integers to be adjusted, and let $T$ be the infinite spherically symmetric plane tree in which every node $u$ at height $n$ has $2\kappa_{n}$ children, the nodes $u1,\ldots,u(2\kappa_n)$ in the Neveu notation.\footnote{A vertex is labeled as a string $i_1i_2\ldots i_j$, meaning it is the $i_j$'th child of vertex $i_1i_2\ldots i_{j-1}$.}
For each $i\in\llbracket1,2\kappa_n\rrbracket$, we set the length of the edge between $u$ and $ui$ to be $2^{-(n+1)}$ if $i$ is odd, and $2\cdot2^{-(n+1)}$ is $i$ is even, and we denote by $d$ the metric on $T$ induced by these edge-lengths. 
If we complete $T$ into $\overline{T}=T\sqcup\partial T$ by adding its boundary $\partial T$, which consists of rays $u=(u_k)_{k\in\mathbb{N}}$ emanating from the root $\varnothing$ , then $\left(\overline{T},d\right)$ is compact. 
Now, the distribution of a non-backtracking random walk on $T$ starting at the root yields a natural Borel probability measure $\pi$ on $\overline{T}$, supported on $\partial T$.
Let $\xi_1$ and $\xi_2$ be independent random variables with distribution $\pi$, and let us consider the distribution of $d(\varnothing,\xi_1)$ and $d(\varnothing,\xi_2)$, and of $d(\xi_1,\xi_2)$.
Recall the following well known fact: if $B_1,B_2,\ldots$ are independent Bernoulli random variables with success probability $1/2$, then the random variable $\sum_{n\geq1}B_n\cdot2^{-n}$ has uniform distribution on $[0,1]$.
Now the random variables $d(\varnothing,\xi_1)$ and $d(\varnothing,\xi_2)$ have the same distribution as
\[\sum_{n\geq0}(1+\mathbf{1}(\text{$I_n$ is even}))\cdot2^{-(n+1)},\]
where $(I_n)_{n\in\mathbb{N}}$ is a sequence of independent random variables such that for each $n\in\mathbb{N}$, the random variable $I_n$ has uniform distribution on $\llbracket1,2\kappa_n\rrbracket$.
Therefore, the random variables $d(\varnothing,\xi_1)$ and $d(\varnothing,\xi_2)$ have the same distribution as $1+U$, where $U$ has uniform distribution on $[0,1]$.

We can continue this argument (details deferred) to show
\begin{Lemma}
\label{L:xixi}
 The random variable $d(\xi_1,\xi_2)$ has a continuous probability density function $\Psi$ given by
\begin{equation}\label{eq:pdfPsi}
\Psi(x)=\sum_{n\geq0}\frac{2\kappa_n-1}{(2\kappa_0)\cdot\ldots\cdot(2\kappa_n)}\cdot\psi_n(x),
\end{equation}
where for each $n\in\mathbb{N}$, the function $\psi_n$ is a continuous probability density function supported in $\left[4\cdot2^{-(n+1)},4\cdot2^{-n}\right]$, and bounded by $2^{n+1}$ over this interval.
\end{Lemma}
Granted Lemma \ref{L:xixi}, we complete the proof of Proposition \ref{prop:petit0} as follows.
Choose $(\kappa_n)_{n\in\mathbb{N}}$ so that
\[(2\kappa_0)\cdot\ldots\cdot(2\kappa_n)\geq F\left(4\cdot2^{-(n+2)}\right)^{-1}\quad\text{for all $n\in\mathbb{N}$.}\]
This is certainly possible, it suffices to be very crude and choose 
\[\kappa_n\geq F\left(4\cdot2^{-(n+2)}\right)^{-1}\quad\text{for all $n\in\mathbb{N}$.}\]
This way, for each $n\in\mathbb{N}^*$, we have, for all $\varepsilon\in\left[4\cdot2^{-(n+1)},4\cdot2^{-n}\right]$,
\[\begin{split}
\int_0^\varepsilon\Psi(x)\mathrm{d}x&\leq\sum_{p\geq n}\frac{2\kappa_p-1}{(2\kappa_0)\cdot\ldots\cdot(2\kappa_p)}\\
&=\sum_{p\geq n}\left(\frac{1}{(2\kappa_0)\cdot\ldots\cdot(2\kappa_{p-1})}-\frac{1}{(2\kappa_0)\cdot\ldots\cdot(2\kappa_p)}\right)\\
&=\frac{1}{(2\kappa_0)\cdot\ldots\cdot(2\kappa_{n-1})}\\
&\leq F\left(4\cdot2^{-(n+1)}\right)\leq F(\varepsilon),
\end{split}\]
hence $\int_0^\varepsilon\Psi(x)\mathrm{d}x\leq F(\varepsilon)$ for all $\varepsilon\in{]0,2]}$.
\end{proof}

\begin{proof}[Proof of Lemma \ref{L:xixi}]
For each $n\in\mathbb{N}$, we have
\[\mathbb{P}(|\xi_1\wedge\xi_2|=n)=\frac{2\kappa_n-1}{(2\kappa_0)\cdot\ldots\cdot(2\kappa_n)},\]
and conditionally on $(|\xi_1\wedge\xi_2|=n)$, the random variable $d(\xi_1,\xi_2)$ has the same distribution as
\begin{multline*}
\sum_{p\geq n+1}\left(1+\mathbf{1}\left(\text{$I^1_p$ is even}\right)\right)\cdot2^{-(p+1)}+\left(1+\mathbf{1}\left(\text{$I^1_n$ is even}\right)\right)\cdot2^{-(n+1)}\\
+\left(1+\mathbf{1}\left(\text{$I^2_n$ is even}\right)\right)\cdot2^{-(n+1)}+\sum_{p\geq n+1}\left(1+\mathbf{1}\left(\text{$I^2_p$ is even}\right)\right)\cdot2^{-(p+1)}
\end{multline*}
conditioned $I^1_n\neq I^2_n$, where $\left(I^1_p\right)_{p\in\mathbb{N}}$ and $\left(I^2_p\right)_{p\in\mathbb{N}}$ are independent sequences of independent random variables such that for each $p\in\mathbb{N}$, the random variables $I^1_p$ and $I^2_p$ have uniform distribution on $\llbracket1,2\kappa_p\rrbracket$.
We simplify this into: conditionally on $(|\xi_1\wedge\xi_2|=n)$, the random variable $d(\xi_1,\xi_2)$ has the same distribution as
\begin{multline*}
(1+U_1)\cdot2^{-(n+1)}+\left(2+\mathbf{1}\left(\text{$I^1_n$ is even}\right)+\mathbf{1}\left(\text{$I^2_n$ is even}\right)\right)\cdot2^{-(n+1)}+(1+U_2)\cdot2^{-(n+1)}\\
=\left(4+\mathbf{1}\left(\text{$I^1_n$ is even}\right)+\mathbf{1}\left(\text{$I^2_n$ is even}\right)+U_1+U_2\right)\cdot2^{-(n+1)}
\end{multline*}
conditioned $I^1_n\neq I^2_n$, where $U_1$ and $U_2$ are independent random variables with uniform distribution on $[0,1]$, independent of $I^1_n$ and $I^2_n$.
Finally, we check that the last distribution has density $\psi_n$ given by
\begin{eqnarray*}
\psi_n(x)&=&\frac{\kappa_n(\kappa_n-1)}{2\kappa_n(2\kappa_n-1)}\cdot2^{n+1}\cdot\phi\left(2^{n+1}\cdot x-4\right)\\
&&+\frac{2\kappa_n^2}{2\kappa_n(2\kappa_n-1)}\cdot2^{n+1}\cdot\phi\left(2^{n+1}\cdot x-5\right)\\
&&+\frac{\kappa_n(\kappa_n-1)}{2\kappa_n(2\kappa_n-1)}\cdot2^{n+1}\cdot\phi\left(2^{n+1}\cdot x-6\right),
\end{eqnarray*}
for all $x\in\mathbb{R}_+$, where $\phi$ is the probability density function of $U_1+U_2$.
The first term accounts for the case where neither $I^1_n$ nor $I^2_n$ is even, the second term for the case where exactly one of them is even, and the last term for the case where both are even.
Note that we have
\[\phi(x)=\begin{cases}
x&\text{if $x\in[0,1]$}\\
2-x&\text{if $x\in[1,2]$}\\
0&\text{otherwise}
\end{cases}\quad\text{for all $x\in\mathbb{R}_+$.}\]
Equation \eqref{eq:pdfPsi} readily follows.
\end{proof}

\section{Consequences}\label{sec:csq}

In this section, we prove Proposition \ref{prop:?} and Corollary \ref{cor:density}.
We start with the proof of Proposition \ref{prop:?}, which relies on the tree constructions of the previous section.

\begin{proof}[Proof of Proposition \ref{prop:?}]
Let $f$ be a continuous probability density function on $\mathbb{R}_+$, and assume that there exists $\eta>0$ such that $f(x)>0$ for all $x\in{]0,\eta]}$.
First, consider the following lemma, which is very similar to Proposition \ref{prop:petit0}.
\begin{Lemma}\label{lem:?}
There exists a pointed compact measured metric space $(X,d,\pi,x_0)$ such that the following holds, where $\xi_1$ and $\xi_2$ are independent random variables with distribution $\pi$.
\begin{itemize}
\item The random variable $d(\xi_1,\xi_2)$ has a continuous probability density function $\Psi$ supported on $[0,4]$ such that for every $\varepsilon>0$, we have $\Psi(x)=o(f(\varepsilon x))$ as $x\to0^+$.
\item The random variables $d(x_0,\xi_1)$ and $d(x_0,\xi_2)$ have the same distribution as $1+U$, where $U$ has uniform distribution on $[0,1]$.
\end{itemize}
\end{Lemma}
\begin{proof}[Proof of the lemma]
Keeping the notation introduced in the proof of Proposition \ref{prop:petit0}, let us pick up the construction of the compact measured metric space $\left(\overline{T},d,\pi\right)$, where the $(\kappa_n)_{n\in\mathbb{N}}$ are  to be chosen later.
The random variable $d(\xi_1,\xi_2)$ has a continuous probability density function $\Psi$ given by \eqref{eq:pdfPsi}, hence such that for each $n\in\mathbb{N}$, we have
\[\Psi(x)\leq\frac{2\kappa_n-1}{(2\kappa_0)\cdot\ldots\cdot(2\kappa_n)}\cdot2^{n+1}\quad\text{for all $x\in\left[4\cdot2^{-(n+1)},4\cdot2^{-n}\right]$.}\]
Now, let us adjust the $(\kappa_n)_{n\in\mathbb{N}}$ so that for each $\varepsilon>0$, we have $\Psi(x)=o(f(\varepsilon x))$ as ${x\to0^+}$.
Note that we can reduce the problem to simply requiring that $\Psi(x)=o(g(x))$ as $x\to0^+$, where ${g(x)=\min_{\left[x^2,\eta\right]}f}$.
Indeed, for each $\varepsilon>0$, as $x\mapsto\min_{[x,\eta]}f$ is non-decreasing, we have 
\[g(x)=\min_{\left[x^2,\eta\right]}f\leq\min_{[\varepsilon x]}f\leq f(\varepsilon x)\quad\text{for all sufficiently small $x$.}\]
Therefore, let us choose the $(\kappa_n)_{n\in\mathbb{N}}$ so that
\[\kappa_0\cdot\ldots\cdot\kappa_n\geq\frac{2^{n+1}}{g\left(4\cdot2^{-(n+2)}\right)}\quad\text{for all sufficiently large $n$.}\]
This is certainly possible, it suffices to be very crude and take
\[\kappa_n\geq\frac{2^{n+1}}{g\left(4\cdot2^{-(n+2)}\right)}\quad\text{for all sufficiently large $n$.}\]
This way, for all sufficiently large $n$, we have, for all $x\in\left[4\cdot2^{-(n+1)},4\cdot2^{-n}\right]$,
\[\begin{split}
\Psi(x)&\leq\frac{2\kappa_n-1}{(2\kappa_0)\cdot\ldots\cdot(2\kappa_n)}\cdot2^{n+1}\\
&\leq\frac{2\kappa_n}{(2\kappa_0)\cdot\ldots\cdot(2\kappa_n)}\cdot2^{n+1}\\
&=\frac{2}{\kappa_0\cdot\ldots\cdot\kappa_{n-1}}\\
&\leq4\cdot2^{-(n+1)}\cdot g\left(4\cdot2^{-(n+1)}\right)\leq x\cdot g(x),
\end{split}\]
hence $\Psi(x)\leq x\cdot g(x)$ for all sufficiently small $x$.
To complete the proof of the lemma, we let $(X,d,\pi,x_0)=\left(\overline{T},d,\pi,\varnothing\right)$.
\end{proof}

Now with the $\left(X,d,\pi,x_0\right)$ and $\Psi$ of Lemma \ref{lem:?} at hand, we resume the proof of the proposition.
Fix $\beta\in{]0,1/3[}$, and fix an integer $n\geq4$: using the tree construction of Proposition \ref{prop:treeconstruction} together with Lemma \ref{lem:?}, we will construct a compact measured metric space $(S,d,\mu)$ such that the distribution $\theta$ arising from $(S,d,\mu)$ has a continuous probability density function $g$ on $\mathbb{R}_+$, with $g\leq(1+\zeta(\beta,n))\cdot f$, where $\zeta(\beta,n)\rightarrow0$ as $\beta\to0$ and $n\to\infty$.

The proof is slightly technical.
First, fix $0<\kappa<K$ such that ${\int_\kappa^Kf(x)\mathrm{d}x\geq1-\beta}$.
Then, let $B=\left\{x\in[\kappa,K]:f(x)\geq\beta/K\right\}$, and note that 
\[\int_Bf(x)\mathrm{d}x=\int_\kappa^Kf(x)\mathrm{d}x-\int_{[\kappa,K]\setminus B}f(x)\mathrm{d}x\geq(1-\beta)-K\cdot\frac{\beta}{K}=1-2\beta.\]
Since $f$ is uniformly continuous on $[0,K]$, we can fix $\varepsilon\in{]0,\eta\wedge\kappa]}$ such that 
\[|f(x)-f(y)|\leq\frac{\beta^2}{K}\quad\text{for all $x,y\in[0,K]$ with $|x-y|\leq\varepsilon$.}\]
Next, we claim that it is possible to fix $p_0\in{]0,1[}$ such that
\begin{equation}\label{eq:belowf}
p_0\cdot\Psi(x)\leq\varepsilon/n\cdot f(\varepsilon/n\cdot x)\quad\text{for all $x\in[0,4]$,}
\end{equation}
where $\Psi$ is the continuous probability density function supported on $[0,4]$ provided by Lemma \ref{lem:?}.
Indeed, since ${\Psi(x)=o(f(\varepsilon/n\cdot x))}$ as $x\to0^+$, there exists $x_0\in{]0,4]}$ such that 
\[\Psi(x)\leq\varepsilon/n\cdot f(\varepsilon/n\cdot x)\quad\text{for all $x\in{[0,x_0]}$.}\]
Then, since 
\[\min_{x\in[x_0,4]}f(\varepsilon/n\cdot x)=\min_{[\varepsilon x_0/n,4\varepsilon/n]}f>0\]
(we use here that $4\varepsilon/n\leq\eta$), we can fix $p_0\in{]0,1[}$ such that 
\[p_0\cdot\max_{[0,4]}\Psi\leq\varepsilon/n\cdot\min_{[\varepsilon x_0/n,4\varepsilon/n]}f.\]
Equation \eqref{eq:belowf} readily follows.
Next, let $d_1<\ldots<d_k\in B$ be at least $\varepsilon$ apart from each other, and such that ${[d_1,d_1+\varepsilon]\cup\ldots\cup[d_k,d_k+\varepsilon]\supset B}$.
To find such points, it suffices to take $d_1=\inf B$, and by induction, for every $i\in\mathbb{N}^*$ such that $d_1,\ldots,d_i$ have been constructed, proceed as follows: if ${[d_1,d_1+\varepsilon]\cup\ldots\cup[d_i,d_i+\varepsilon]}$ already covers $B$, then terminate, otherwise, let 
\[d_{i+1}=\inf(B\setminus([d_1,d_1+\varepsilon]\cup\ldots\cup[d_i,d_i+\varepsilon])).\]
At the end of the construction, note that since the intervals $[d_1,d_1+\varepsilon[,\ldots,[d_k,d_k+\varepsilon[$ are disjoint and included in $[\kappa,K+\varepsilon[$, we have $k\cdot\varepsilon\leq(K+\varepsilon)-\kappa\leq K$ (we use here that $\varepsilon\leq\kappa$).
Now, by the definition of $\varepsilon$, for each $i\in\llbracket1,k\rrbracket$, we have ${|f(d_i)-f(x)|\leq\left.\beta^2\middle/K\right.}$ for all $x\in[d_i,d_i+\varepsilon]$, hence
\[\sum_{i=1}^k\varepsilon\cdot f(d_i)\geq\sum_{i=1}^k\left(\int_{d_i}^{d_i+\varepsilon}f(x)\mathrm{d}x-\varepsilon\cdot\frac{\beta^2}{K}
\right)\geq\int_Bf(x)\mathrm{d}x-K\cdot\frac{\beta^2}{K}\geq(1-2\beta)-\beta^2\geq1-3\beta.\]
By the mean value theorem, we deduce that it is possible to fix $\alpha_1,\ldots,\alpha_k\in{]0,1/(1-3\beta)[}$ such that 
\[\alpha_1\cdot\varepsilon\cdot f(d_1)+\ldots+\alpha_k\cdot\varepsilon\cdot f(d_k)=1-p_0,\]
i.e, such that
\[\theta=p_0\cdot\delta_0+\alpha_1\cdot\varepsilon\cdot f(d_1)\cdot\delta_{d_1}+\ldots+\alpha_k\cdot\varepsilon\cdot f(d_k)\cdot\delta_{d_k}\]
is a probability measure.
Now, let us discretise further: we let $J_n=\left\llbracket0,n^2-4n\right\rrbracket$, and consider
\[\theta'=p_0\cdot\delta_0+\sum_{i=1}^k\alpha_i\cdot\varepsilon\cdot f(d_i)\cdot\frac{1}{\#J_n}\sum_{j\in J_n}\delta_{d_i+j\cdot\varepsilon\cdot n^{-2}}.\]
By Proposition \ref{prop:treeconstruction}, there exists a tree structure $(T,\rho,d,\nu)$ that achieves $\theta'$.
We use it to construct our compact measured metric space $(S,d,\mu)$ as follows.
Onto each leaf $u$ of $T$, we graft a copy $\left(X^u,d^u,\pi^u,x_0^u\right)$ of the pointed compact measured metric space $(X,d,\pi,x_0)$ provided by Lemma \ref{lem:?}, scaling distances by $\varepsilon/n$, and identifying the marked point $x_0^u$ with $u$.
The construction is represented in Figure \ref{fig:density}.

\begin{figure}[ht]
\centering
\includegraphics[width=0.6\linewidth]{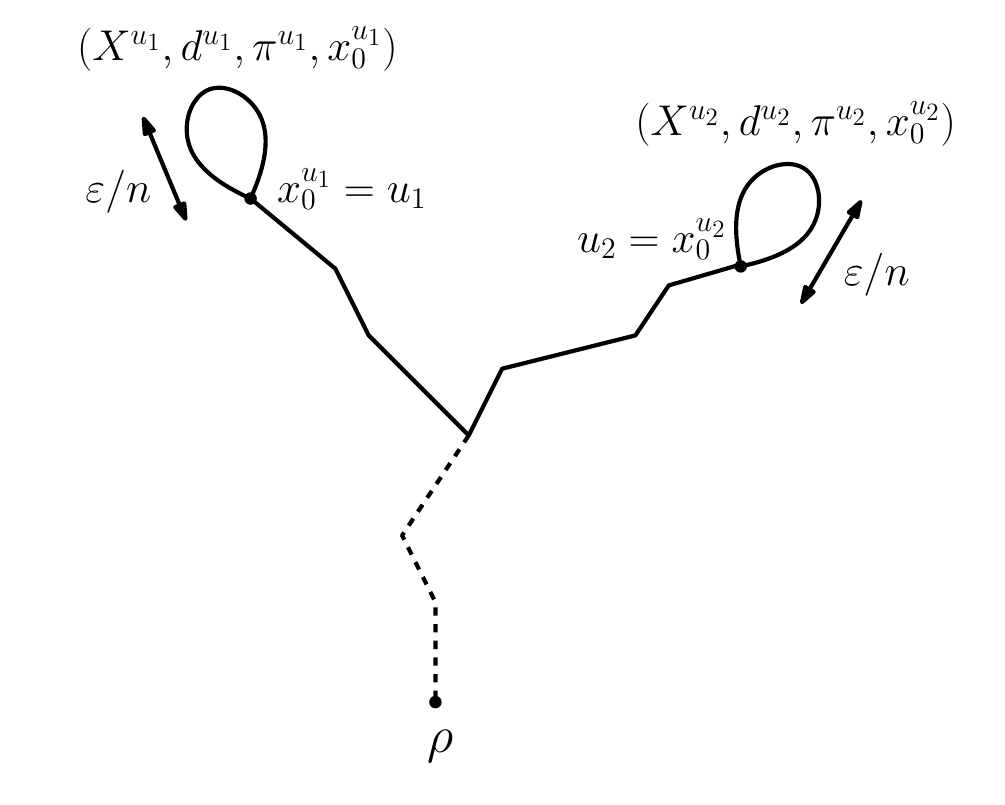}
\caption{A part of the construction of $(S,d,\mu)$. Let $\xi_1$ and $\xi_2$ be independent random variables with distribution $\mu$, and let $u_1\neq u_2$ be leaves of $T$. Conditionally on ${\xi_1\in X^{u_1}}$ and ${\xi_2\in X^{u_2}}$, the random variable $d(\xi_1,\xi_2)$ has the same distribution as ${\varepsilon/n\cdot(1+U_1)+d(u_1,u_2)+\varepsilon/n\cdot(1+U_2)=d(u_1,u_2)+\varepsilon/n\cdot(2+U_1+U_2)}$, where $U_1$ and $U_2$ are independent random variables with uniform distribution on $[0,1]$. Moreover, by construction, we have $d(u_1,u_2)=d_i+j\cdot\varepsilon\cdot n^{-2}$ for some $i\in\llbracket1,k\rrbracket$ and $j\in J_n$.}\label{fig:density}
\end{figure}

We denote by $(S,d,\mu)$ the compact measured metric space obtained in this way, where ${\mu=\sum_{\text{$u$ leaf of $T$}}\nu\{u\}\cdot\pi^u}$.
Now, we claim that, if $\xi_1$ and $\xi_2$ are two independent random variables with distribution $\mu$, then the random variable $d(\xi_1,\xi_2)$ has probability density function $g$ given by
\begin{eqnarray*}
\lefteqn{g(x)}\\
&=&p_0\cdot\frac{\Psi\left(x\cdot(\varepsilon/n)^{-1}\right)}{\varepsilon/n}+\sum_{i=1}^k\alpha_i\cdot\varepsilon\cdot f(d_i)\cdot\frac{1}{\#J_n}\cdot\sum_{j\in J_n}\frac{\phi\left(\left(x-\left(d_i+j\cdot\varepsilon\cdot n^{-2}\right)\right)\cdot(\varepsilon/n)^{-1}-2\right)}{\varepsilon/n}\\
&=&p_0\cdot\frac{\Psi\left(x\cdot(\varepsilon/n)^{-1}\right)}{\varepsilon/n}+\sum_{i=1}^k\alpha_i\cdot f(d_i)\cdot\frac{n}{\#J_n}\cdot\sum_{j\in J_n}\phi\left(\frac{x-d_i}{\varepsilon/n}-\frac{j}{n}-2\right)
\end{eqnarray*}
for all $x\in\mathbb{R}_+$, where $\phi$ is the probability density function of the sum of two independent random variables with uniform distribution on $[0,1]$.
See Figure \ref{fig:density}.

Now, recall that $\Psi$ is supported in $[0,4]$, hence 
\[\Psi\left(\frac{x}{\varepsilon/n}\right)>0\Longrightarrow x\in[0,4\varepsilon/n].\]
Moreover, since $\phi$ is supported in $[0,2]$, we have, for each $i\in\llbracket1,k\rrbracket$, by the definition of $J_n$:
\[\sum_{j\in J_n}\phi\left(\frac{x-d_i}{\varepsilon/n}-\frac{j}{n}-2\right)>0\Longrightarrow x\in[d_i,d_i+\varepsilon].\]
Since $4\varepsilon/n\leq\kappa\leq d_1$, the interval $[0,4\varepsilon/n[$ and the intervals ${({[d_i,d_i+\varepsilon[}\,;\,i\in\llbracket1,k\rrbracket)}$ are mutually disjoint: thus, we have
\[g(x)=\begin{cases}
p_0\cdot\frac{\Psi\left(x\cdot(\varepsilon/n)^{-1}\right)}{\varepsilon/n}&\text{if $x\in[0,4\varepsilon/n]$}\\
\alpha_i\cdot f(d_i)\cdot\frac{n}{\#J_n}\cdot\sum_{j\in J_n}\phi\left(\frac{x-d_i}{\varepsilon/n}-\frac{j}{n}-2\right)&\text{if $x\in[d_i,d_i+\varepsilon]$}\\
0&\text{otherwise.}
\end{cases}\quad\text{for all $x\in\mathbb{R}_+$}\]
To conclude the proof, we compare $g$ with $f$ over each one of these intervals.
First, by \eqref{eq:belowf}, we have
\[g(x)=p_0\cdot\frac{\Psi\left(x\cdot(\varepsilon/n)^{-1}\right)}{\varepsilon/n}\leq f(x)\quad\text{for all $x\in[0,4\varepsilon/n]$.}\]
Then, for each $i\in\llbracket1,k\rrbracket$, we have
\begin{equation}\label{eq:boundoverinti}
\begin{split}
g(x)&=\alpha_i\cdot f(d_i)\cdot\frac{n}{\#J_n}\cdot\sum_{j\in J_n}\phi\left(\frac{x-d_i}{\varepsilon/n}-\frac{j}{n}-2\right)\\
&\leq\frac{1}{1-3\beta}\cdot\frac{f(x)}{1-\beta}\cdot\frac{n}{\#J_n}\cdot\max_{x\in[d_i,d_i+\varepsilon]}\sum_{j\in J_n}\phi\left(\frac{x-d_i}{\varepsilon/n}-\frac{j}{n}-2\right)
\end{split}
\end{equation}
for all $x\in[d_i,d_i+\varepsilon]$, where the inequality for the second term comes from rearranging
\[\left|f(d_i)-f(x)\right|\leq\frac{\beta^2}{K}=\beta\cdot\frac{\beta}{K}\leq\beta\cdot f(d_i).\]
Finally, let us bound
\[\max_{x\in[d_i,d_i+\varepsilon]}\sum_{j\in J_n}\phi\left(\frac{x-d_i}{\varepsilon/n}-\frac{j}{n}-2\right)=\max_{y\in[0,1]}\sum_{j\in J_n}\phi\left(ny-\frac{j}{n}-2\right).\]
For every $y\in[0,1]$, we have
\[\begin{split}
\sum_{j\in J_n}\phi\left(ny-\frac{j}{n}-2\right)&=\sum_{j\in J_n}n\cdot\int_{j/n}^{(j+1)/n}\phi\left(ny-\frac{\lfloor nz\rfloor}{n}-2\right)\mathrm{d}z\\
&\leq n\cdot\int_{-\infty}^\infty\phi\left(ny-\frac{\lfloor nz\rfloor}{n}-2\right)\mathrm{d}z.
\end{split}\]
Support considerations show that
\[\begin{split}
\int_{-\infty}^\infty\phi\left(ny-\frac{\lfloor nz\rfloor}{n}-2\right)\mathrm{d}z=\int_{ny-4}^{ny-2+1/n}\phi\left(ny-\frac{\lfloor nz\rfloor}{n}-2\right)\mathrm{d}z.
\end{split}\]
Since $|\lfloor nz\rfloor/n-z|\leq1/n$ for all $z$, using the $1$-Lispchitz continuity of $\phi$, we get
\[\begin{split}
\int_{ny-4}^{ny-2+1/n}\phi\left(ny-\frac{\lfloor nz\rfloor}{n}-2\right)\mathrm{d}z&\leq\int_{ny-4}^{ny-2+1/n}\left(\phi(ny-z-2)+\frac{1}{n}\right)\mathrm{d}z\\
&=\int_{-1/n}^2\phi(t)\mathrm{d}t+\frac{2+1/n}{n}\leq1+\frac{3}{n}.
\end{split}\]
Thus, we obtain
\[\phi\left(ny-\frac{j}{n}-2\right)\leq n\cdot\left(1+\frac{3}{n}\right)\quad\text{uniformly in $y\in[0,1]$.}\]
Plugging this into \eqref{eq:boundoverinti}, we end up with 
\[\begin{split}
g(x)&\leq\frac{1}{1-3\beta}\cdot\frac{f(x)}{1-\beta}\cdot\frac{n}{\#J_n}\cdot n\cdot\left(1+\frac{3}{n}\right)\\
&=\frac{1}{1-3\beta}\cdot\frac{1}{1-\beta}\cdot\frac{n^2}{n^2-4n+1}\cdot\left(1+\frac{3}{n}\right)\cdot f(x)
\end{split}\]
for all $x\in[d_i,d_i+\varepsilon]$.
Since this multiplicative constant depends onlly on $\beta$ and $n$, and goes to $1$ as $\beta\to0$ and $n\to\infty$, the proof is complete.
\end{proof}

Finally, we show how Proposition \ref{prop:?} implies that every distribution with a suitable density is achieved by some \textbf{random} compact measured metric space (Corollary \ref{cor:density}).
Although we will not need sophisticated Gromov--Hausdorff--Prokhorov theory, let us take a brief paragraph to set the scene more rigorously.
Let $\mathbb{M}$ be the space of compact measured metric spaces $(S,d,\mu)$ (i.e, compact metric spaces $(S,d)$ endowed with a Borel probability measure $\mu$), seen up to measure-preserving isometry.
We denote by $[S,d,\mu]$ the equivalence class of $(S,d,\mu)$.
The Gromov--Hausdorff--Prokhorov metric $D:\mathbb{M}\times\mathbb{M}\rightarrow\mathbb{R}_+$ is defined by
\[D([S_1,d_1,\mu_1],[S_2,d_2,\mu_2])=\inf_{\substack{\phi_1:S_1\rightarrow S\\\phi_2:S_2\rightarrow S}}d_\mathrm{H}(\phi_1(S_1),\phi_2(S_2))\vee d_\mathrm{P}(\phi_1^*\mu_1,\phi_2^*\mu_2)\]
for all $[S_1,d_1,\mu_1],[S_2,d_2,\mu_2]\in\mathbb{M}$, where the infimum is over all isometric embeddings ${\phi_1:S_1\rightarrow S}$ and ${\phi_2:S_2\rightarrow S}$ into some common metric space $(S,d)$, and where $d_\mathrm{H}$ denotes the Hausdorff distance between non-empty compact subsets of $S$, and $d_\mathrm{P}$ the Prokhorov distance between Borel probability measures on $S$.
This makes $(\mathbb{M},D)$ into a separable and complete metric space, that we equip with its Borel $\sigma$-algebra.
Then, the fact that the mapping ${\kappa:\mathbb{M}\times\mathcal{B}(\mathbb{R}_+)\rightarrow[0,1]}$ defined by
\[\kappa([S,d,\mu],B)=\int_S\int_S\mathbf{1}(d(s_1,s_2)\in B)\mathrm{d}\mu(s_2)\mathrm{d}\mu(s_1)
\quad\text{for all $[S,d,\mu]\in\mathbb{M}$ and $B\in\mathcal{B}(\mathbb{R}_+)$}\]
is a Markov kernel allows to give a rigorous definition to sampling a random compact measured metric space $[S,d,\mu]$, and then conditionally on $[S,d,\mu]$, sampling a random variable distributed as $d(\xi_1,\xi_2)$, where $\xi_1$ and $\xi_2$ are independent random variables with distribution $\mu$ conditionally on $[S,d,\mu]$.

\smallskip

Without further ado, let us present the proof of Corollary \ref{cor:density}.

\begin{proof}[Proof of Corollary \ref{cor:density}]
Using Proposition \ref{prop:?} iteratively, we can write
\begin{equation}\label{eq:mixture}
f=\sum_{n\geq1}2^{-n}\cdot g_n\quad\text{almost everywhere,}
\end{equation}
where for each $n\in\mathbb{N}^*$, the function $f_n:\mathbb{R}_+\rightarrow\mathbb{R}_+$ is the probability density function of a distribution which is achieved by some compact measured metric space $(S_n,d_n,\mu_n)$.
Then, to achieve the distribution with density $f$, it suffices to take a mixture of the $(S_n,d_n,\mu_n)$, with weights $2^{-n}$.
To complete the proof, let us justify \eqref{eq:mixture}.
Let $n\in\mathbb{N}$, and assume by induction that there are continuous probability density functions ${g_1,\ldots,g_n}$ and $f_n$ on $\mathbb{R}_+$ such that
\[f=\sum_{k=1}^n2^{-k}\cdot g_k+2^{-n}\cdot f_n,\]
where for each $k\in\llbracket1,n\rrbracket$, the distribution with density $g_k$ is achieved by some compact measured metric space $(S_k,d_k,\mu_k)$, and where there exists $\eta>0$ such that $f_n(x)>0$ for all $x\in{]0,\eta]}$.
(Initially, consider the trivial decomposition ${f=f_0}$.)
Then, by Proposition \ref{prop:?}, there exists a compact measured metric space $(S_{n+1},d_{n+1},\mu_{n+1})$ such that the distribution $\theta$ arising from $(S_{n+1},d_{n+1},\mu_{n+1})$ has a continuous density $g_{n+1}$ on $\mathbb{R}_+$, where $g_{n+1}\leq 3/2\cdot f_n$.
Next, we can write
\[f_n=\frac{1}{2}\cdot g_{n+1}+\frac{1}{2}\cdot f_{n+1},\]
where
\[f_{n+1}=\frac{f_n-g_{n+1}/2}{1-1/2}.\]
By construction, we have
\[f=\sum_{k=1}^{n+1}2^{-k}\cdot g_k+2^{-(n+1)}\cdot f_{n+1},\]
and since
\[f_{n+1}\geq\frac{f_n/4}{1-1/2}=\frac{f_n}{2},\]
we have $f_{n+1}(x)>0$ for all $x\in{]0,\eta]}$.
By induction, we obtain in this way a sequence of continuous probability density functions $g_1,g_2,\ldots$ on $\mathbb{R}_+$ such that for each $k\in\mathbb{N}^*$, the distribution with density $g_k$ is achieved by some compact measured metric space $(S_k,d_k,\mu_k)$, and where
\[f\geq\sum_{k\geq1}2^{-k}\cdot g_k.\]
As both terms integrate to $1$, they must be equal almost everywhere, which yields \eqref{eq:mixture}.
\end{proof}

%
%
%
%
%
%
%
%
%
%
%
%
%
%
%
%
%
%

\paragraph{Acknowledgements.}
The counter-example in Proposition \ref{prop:mathoverflow} was outlined by ``Mike" in response to our question posted on math overflow\footnote{\href{https://mathoverflow.net/questions/428539/distributions-of-distance-between-two-random-points-in-hilbert-space}{https://mathoverflow.net/questions/428539/distributions-of-distance-between-two-random-points-in-hilbert-space}}, and after comments by Iosif Pinelis, the present version was given on math overflow by Yuval Peres.
The second and third author were supported by the ERC Consolidator Grant ``SuPerGRandMa'' 101087572.

   
 \end{document}